\documentclass[12pt]{amsart}
\usepackage{amssymb,latexsym,amsmath,amscd,amsthm,graphicx, color}
\usepackage[all]{xy}
\usepackage{pgf,tikz}
\usepackage{mathrsfs}
\usetikzlibrary{arrows}
\definecolor{uuuuuu}{rgb}{0.26666666666666666,0.26666666666666666,0.26666666666666666}
\definecolor{xdxdff}{rgb}{0.49019607843137253,0.49019607843137253,1.}
\definecolor{ffqqqq}{rgb}{1.,0.,0.}

\definecolor{uuuuuu}{rgb}{0.26666666666666666,0.26666666666666666,0.26666666666666666}
\definecolor{qqwuqq}{rgb}{0.,0.39215686274509803,0.}
\definecolor{zzttqq}{rgb}{0.6,0.2,0.}
\definecolor{xdxdff}{rgb}{0.49019607843137253,0.49019607843137253,1.}
\definecolor{qqqqff}{rgb}{0.,0.,1.}
\definecolor{cqcqcq}{rgb}{0.7529411764705882,0.7529411764705882,0.7529411764705882}
\definecolor{sqsqsq}{rgb}{0.12549019607843137,0.12549019607843137,0.12549019607843137}

\theoremstyle{plain}
\newtheorem{example}[subsection]{Example}

\newtheorem{theorem}[subsection]{Theorem}

\newtheorem{lemma}[subsection]{Lemma}

\newtheorem{prop}[subsection]{Proposition}

\theoremstyle{definition}

\newtheorem{remark}[subsection]{Remark}

\newtheorem{note}[subsection]{Note}

\newcommand{\uu}{\cup}
\newcommand{\ii}{\cap}
\newcommand{\UU}{\bigcup}
\newcommand{\II}{\bigcap}

\newcommand{\sci}{\subset}

\newcommand{\es}{\emptyset}
\newcommand{\set}[1]{\{#1\}}

\newcommand{\ga}{\alpha}

\newcommand{\gn}{\nu}
\newcommand{\go}{\omega}

\newcommand{\gs}{\sigma}

\newcommand{\gG}{\Gamma}

\newcommand{\tit}{\textit}

\newcommand{\C}[1]{\mathcal{#1}}
\newcommand{\D}[1]{\mathbb{#1}}
\newcommand{\F}[1]{\mathfrak{#1}}

\newcommand{\te}{\text}

\newcommand{\ep}{\epsilon}

\newcommand{\ol}{\overline}
\newcommand{\ul}{\underline}

\begin{document}

\title{Quantization dimension and stability for infinite self-similar measures with respect to geometric mean error}

\author{Mrinal K. Roychowdhury}
\address{School of Mathematical and Statistical Sciences, University of Texas Rio Grande Valley, 1201 West University Drive, Edinburg, TX 78539-2999, USA}
\email{mrinal.roychowdhury@utrgv.edu}

\author{S. Verma}
\address{Department of Mathematics, IIT Delhi, New Delhi, India 110016 }
\email{saurabh331146@gmail.com}






\subjclass[2010]{28A33, 28A80, 60B05, 94A34.}
\keywords{Infinite IFS, self-similar measure, quantization dimension, Hausdorff dimension, stability.}

\thanks{The second author expresses his gratitude to the University Grants
   Commission (UGC), India, for financial support in the form of a research fellowship.
}
\date{}
\maketitle

\pagestyle{myheadings}\markboth{Mrinal K. Roychowdhury and S. Verma}{Quantization dimension and stability for infinite self-similar measures}

\begin{abstract}
Let $\mu$ be a Borel probability measure associated with an iterated function system consisting of a countably infinite number of contracting similarities and an infinite probability vector.
In this paper, we study the quantization dimension of the measure $\mu$ with respect to the geometric mean error. The quantization for infinite systems is different  from the well-known finite case investigated by Graf and Luschgy. That is, many tools which are used in the finite setting, for example, existence of finite maximal antichains, fail in the infinite case. We prove that the quantization dimension of the measure $\mu$ is equal to its Hausdorff dimension which extends a well-known result of Graf and Luschgy for the finite case to an infinite setting. In the last section, we discuss the stability of quantization dimension for infinite systems.

\end{abstract}


.

\section{Introduction}
The theory of quantization studies the process of approximating probability measures with discrete probability measures supported on a finite set.  This problem stems from the theory of signal processing and finds applications in many areas, for examples, economics, statistics, numerical integration (see \cite{BW, GG, GN, P, Z} for more details). Rigorous mathematical treatment of the quantization theory is given in Graf-Luschgy's book (see \cite{GL1}). The quantization error is often measured  with respect to $L_r$-metrics, where $0<r<+\infty$, and the order of convergence to zero of the quantization error has been investigated in details for different probability measures including invariant measures. Another complimentary approach to quantization is to investigate the order of convergence to zero of the geometric mean error. This is motivated by the fact that $L_r$-metric converges to the geometric mean as $r\to 0$. More precisely, for a given Borel probability measure $\mu$ on $\D R^k$, the \tit{$n$th quantization error} for $\mu$ with respect to the geometric mean error is defined by
\begin{equation}\label{eq00}
 e_n(\mu):=\inf \Big\{\exp \int \log d(x, \ga) d\mu(x) : \ga \sci \D R^k, \, 1\leq\te{Card}(\ga) \leq n\Big\},
\end{equation}
where $d(x, \ga)$ denotes the distance between $x$ and the set $\ga$ with respect to an arbitrary norm $d$ on $\D R^k$. The set $\ga \sci \D R^k$ for which the infimum in \eqref{eq00} is attained is called an \tit{optimal set of $n$-means}, and the collection of all optimal sets of $n$-means for $\mu$ is denoted by $\C C_n(\mu)$. Under some suitable restriction $e_n(\mu)$ tends to zero as $n$ tends to infinity. Let
\begin{equation*}
\hat e_n:=\hat e_n(\mu)=\log e_n(\mu)=\inf \Big\{\int \log d(x, \ga) d\mu(x) : \ga \sci \D R^k, \, 1\leq\te{Card}(\ga) \leq n\Big\}.
\end{equation*}
Then, the following quantities introduced in \cite{GL2}:
\[
\ul D(\mu):=\liminf_{n\to \infty}  \frac{\log n}{-\hat e_n(\mu)} \te{ and } \ol D(\mu):=\limsup_{n\to \infty} \frac{\log n}{-\hat e_n(\mu)},
\]
are called the \tit{lower} and the \tit{upper quantization dimensions} of $\mu$ (of order zero), respectively. If $\ul D (\mu)=\ol D (\mu)$, then the common value is called the quantization dimension of $\mu$ and is denoted by $D(\mu)$. The quantization dimension measures the speed at which the specified measure of the error tends to zero as $n$ tends to infinity. The quantization dimension with respect to the geometric mean error can be regarded as a limit state of that based on $L_r$-metrics as $r$ tends to zero (see \cite[Lemma 3.5]{GL2}). The proposition below plays an important role in estimating the lower and upper quantization dimensions.
\begin{prop} (see \cite[Proposition 4.3]{GL2}) \label{prop1}
Let $\ul D=\ul D(\mu)$ and $\ol D=\ol D(\mu)$.

$(a)$ If $0\leq t<\ul D<s$, then
\[
\lim_{n\to \infty} (\log n +t\hat e_n(\mu))=+\infty, \te{ and } \liminf_{n\to \infty} (\log n+s\hat e_n(\mu)) =-\infty.
\]

$(b)$ If $0\leq t<\ol D<s$, then
\[
\limsup_{n\to \infty} (\log n +t\hat e_n(\mu))=+\infty, \te{ and } \lim_{n\to \infty} (\log n+s\hat e_n(\mu)) =-\infty.
\]
\end{prop}

Let $M$ denote either the set $\set{1, 2, \cdots, N}$ for some positive integer $N\geq 2$, or the set $\D N$ of natural numbers. A collection $\set{S_j : j\in M}$ of similarity mappings on $\D R^k$ with similarity ratios $\set{s_j : j\in M}$ is contractive if $\sup\set{s_j : j\in M}<1$. If $J$ is the limit set of the iterated function system, then it is known that $J$ satisfies the following invariance relation (see \cite{H, MaU, M}):
\[J=\UU_{j\in M} S_j(J).\]
Note that the fractal limit set $J$ of the infinite iterated function system is not necessarily compact, in contrast to the finite case (see \cite{MaU, M}). The iterated function system $\set{S_j: j \in M}$ satisfies the \tit{open set condition} (OSC),  if there exists a bounded nonempty open set $U \sci \D R^k$ such that $S_j(U) \sci U$ for all $j\in M$, and $S_i(U) \II S_j(U) =\es$ for $i, j\in M$ with $i\neq j$. The iterated system is said to satisfy the \tit{strong open set condition} (SOSC) if there is an open set $U$ as above,  so that $U\ii J \neq \es$. The system is said to satisfy the \tit{strong separation condition} (SSC) if $S_i(J) \ii S_j(J)=\es$ for $i, j\in M$ with $i\neq j$, where $J$ is the limit set of the iterated function system $\set{S_j : j \in M}$. If $\set{S_j : j\in M}$ satisfies the strong separation condition then - as is easily seen - it also satisfies the strong open set condition. Let $\mu$ be the self-similar measure generated by a finite system of self-similar mappings $\set{S_1, S_2, \cdots, S_N}$ on $\D R^k$ associated with a probability vector $(p_1, p_2, \cdots, p_N)$. Then, $\mu$ has support the compact set $J$ that satisfy the following conditions:
\begin{equation*} \mu=\sum_{j=1}^N p_j \mu \circ S_j^{-1} \te{ and } J=\UU_{j=1}^N S_j(J).\end{equation*}

The finite case is fairly well understood by now. Graf and Luschgy in their seminal work on quantization, determined under the strong separation condition, the quantization dimension with respect to the geometric mean error of a self-similar measure generated by a finite system of self-similar mappings, and proved that the quantization dimension in this case coincides with the Hausdorff dimension of the measure (see \cite{GL2}). Under the same condition, Zhu determined the quantization dimension with respect to the geometric mean error of a self-conformal measure generated by a finite system of conformal mappings and proved that the quantization dimension in this case also coincides with the Hausdorff dimension of the measure (see \cite{Z1}). Recently, assuming the same separation property, the quantization dimension $D(\mu)$ of a recurrent self-similar measure $\mu$ with respect to the geometric mean error was determined, and it was proved that $D(\mu)$ coincides with the Hausdorff dimension  $\te{dim}_{\te{H}}^\ast(\mu)$ of $\mu$ (see \cite{RS1}). For other work in this direction one could also see \cite {RS2, Z2, Z3}. To determine the quantization dimension of a fractal probability measure in each of the aforementioned case the number of mappings considered in the iterated function system was finite. How to determine the quantization dimension with respect to the geometric mean error of a fractal probability measure generated by an infinite iterated function system was a long-time open problem. This work in this paper is the first advance in this direction and complements the results in \cite{MR2}, where the upper and lower quantization dimensions of order $r \in (0, +\infty)$ of a fractal probability measure generated by an infinite iterated function system were investigated.

Let  $\set{S_j : j\in \D N}$ be an infinite system of self-similar mappings on $\D R^k$ with contractive ratios $\set{s_j : j\in \D N}$ such that $\sup\set{s_j : j\in \D N}<1$. Let $(p_1, p_2, \cdots)$ be an infinite probability vector, with $p_j>0$ for all $j\geq 1$. Then, there exists a unique Borel probability measure $\mu$ supported on $\ol J$ 
(see \cite{H}, \cite{MaU}, \cite{M}, etc.), such that
\begin{equation*} \label{eq10} \mu=\sum_{j=1}^\infty p_j \mu \circ S_j^{-1}.\end{equation*} 
Under the strong open set condition it is known that if the series $\sum_{j=1}^\infty p_j\log s_j$ converges (see \cite{MR1}), then
\[\te{dim}_{\te{H}}^\ast(\mu)=\te{Dim}_{\te{P}}^\ast(\mu)=\frac{\sum_{j=1}^\infty p_j\log p_j}{\sum_{j=1}^\infty p_j\log s_j},\]
where $\te{dim}_{\te{H}}^\ast(\mu)$ and $\te{Dim}_{\te{P}}^\ast(\mu)$ represent the Hausdorff and packing dimensions of the measure $\mu$. In this paper, in Section~\ref{sec3}, under the strong separation condition, if the series $\sum_{j=1}^\infty p_j\log s_j$ converges, then we prove that the quantization dimension $D(\mu)$ of the infinite self-similar measure $\mu$ exists, and satisfies the following condition:
\[D(\mu)=\te{dim}_{\te{H}}^\ast(\mu)=\te{Dim}_{\te{P}}^\ast(\mu)=\frac{\sum_{j=1}^\infty p_j\log p_j}{\sum_{j=1}^\infty p_j\log s_j},\]
that is, the quantization dimension of the infinite self-similar measure $\mu$ coincides with the Hausdorff and packing dimensions of the measure $\mu$. In Section~\ref{sec4}, we show that the quantization dimension of an infinite self-similar measure continuously depends on the parameters such as the infinite probability vector $p=(p_1,p_2,\dots)$ and the infinite function system $\{S_j: j \in \mathbb{N}\}$ involved in its construction, which we term the \tit{stability of the quantization dimension}. To avoid the confusion, note that stability of quantization dimension was also determined by Kesseb\"ohmer and Zhu in \cite{KZ}, where they studied the so-called finite and countable stability of upper quantization dimension, but in our paper by the stability of the quantization dimension we mean how quantization dimension of a measure varies with the parameters defining it, and thus, our definition is different from that given by Kesseb\"ohmer and Zhu (see \cite{KZ}).

\section{Basic Preliminaries} \label{sec2}
Let $X$ be a nonempty compact subset of $\D R^k$ with $X=\te{cl(int} (X))$.
Let $(S_j)_{j=1}^\infty$ be an infinite set of contractive similarity mappings on $X$ with contraction ratios $(s_j)_{j=1}^\infty$, respectively, i.e., $d(S_j(x), S_j(y))=s_j d(x,y)$ for all $x, y \in X$, $0<s_j<1$, $j\geq 1$, and $$0<\inf_{j\ge 1} s_j \leq s:=\sup_{j\ge 1} s_j<1.$$
A \textit{word} with $n$ letters in $\D N =\{1, 2, \ldots\}$,  $\go:=\go_1\go_2\cdots \go_n \in \D N^n$, is said to have \textit{length} $n$, for $n\geq 1$. Define $\D N^{fin}:=\UU_{n\geq 1}\D N^n$ to be the set of all finite words with letters in $\D N$, of any length. For $\go=\go_1\go_2\cdots \go_n \in \D N^n$, define:
\[S_\go=S_{\go_1}\circ S_{\go_2}\circ \cdots \circ S_{\go_n} \te{ and } s_\go=s_{\go_1}s_{\go_2}\cdots s_{\go_n}.\] The empty word $\es$ is the only word of length $0$ and $S_\es=\te{Id}_X$, where $\te{Id}_X$ is the identity mapping on $X$.
For  $\go \in \D N^{fin}\uu \D N^\infty$ and for a positive integer $n$ smaller than the length of $\go$, we denote by $\go|_n$ the word $\go_1\go_2\cdots \go_n$. Notice  that given $\go \in \D N^\infty$, the compact sets $S_{\go|_n}(X)$, $n\geq 1$, are decreasing and their diameters converge to zero. In fact, we have
\begin{equation} \label{eq113} \te{diam}(S_{\go|_n}(X))=s_{\go_1}s_{\go_2}\cdots s_{\go_n}\te{diam}(X)\leq s^n \te{diam}(X). \end{equation}
Hence, for an infinite word $\omega$, the set $\pi(\go):=\II_{n=1}^\infty S_{\go|_n}(X)$
is a singleton, and we can define a map $\pi : \D N^\infty \to X$ which, in view of \eqref{eq113} is continuous. One obtains then the following limit set for the above infinite system of similarities:
\[J:=\pi(\D N^\infty)=\UU_{\go \in \D N^\infty}\II_{n= 1}^\infty S_{\go|_n}(X).\]
This fractal limit set $J$ is not necessarily compact in the infinite case, in contrast to the finite case (see \cite{MaU}).
 Let $\gs : \D N^\infty \to \D N^\infty$ be the shift map on $\D N^\infty$, i.e., $\gs(\go)=\go_2\go_3\cdots$, where $\go=\go_1\go_2\cdots$. Note that
$\pi\circ \gs(\go)=S_{\go_1}^{-1}\circ \pi(\go)$, and hence, rewriting $\pi(\go)=S_{\go_1}(\pi(\gs(\go)))$, we see that $J$ satisfies the invariance condition:
\[J=\UU_{j=1}^\infty S_j(J).\]
In the infinite case, the open set condition and the strong open set condition are not equivalent, unlike in the finite case (see \cite{SW}). In the current paper, we assume that the infinite set of similarities satisfies the \tit{strong separation condition} (SSC), i.e., $S_i(J)\ii S_j(J)=\es$ for $i, j\in \D N$ with $i\neq j$. It is a well-known fact that if an iterated function system satisfies the SSC, then it also satisfies the SOSC.
Let $p:=(p_1, p_2, \cdots)$ be an infinite probability vector, with $p_j>0$ for all $j\geq 1$. Then, there exists a unique Borel probability measure $\mu$ on $\D R^k$ (see \cite{H}, \cite{MaU}, \cite{M}, etc.), such that
\begin{equation*} \label{eq10} \mu=\sum_{j=1}^\infty p_j \mu \circ S_j^{-1}.\end{equation*}
This measure $\mu$ is  called the \tit{self-similar measure} induced by the infinite iterated function system of self-similar mappings $(S_j)_{j\geq 1}$ and by the infinite probability vector $(p_1, p_2, \cdots)$. Let $\gn$ denote the infinite-fold product probability measure on $\D N^\infty$, i.e., $\gn:= p\times p\times p\times \cdots$. Then, $\mu$ stands for the image measure of the measure $\gn$ under the coding map $\pi$, i.e., $\mu=\gn\circ \pi^{-1}$.
One defines the boundary at infinity $\C S(\infty)$ as the set of accumulation points of sequences of type $(S_{i_j}(x_{i_j}))_j$, for distinct integers $i_j$ (see \cite{MaU}).
The self-similar measure $\mu$ is supported in the closure $\ol J$ of the limit set $J$, which is given by $\ol J = J \cup \mathop{\cup}\limits_{\omega \in \D N^{fin}} S_\omega(\C S(\infty))$. In the sequel, we assume that the series $\sum_{j=1}^\infty p_j\log s_j$ converges, and then
\[\te{dim}_{\te{H}}^\ast(\mu)=\te{Dim}_{\te{P}}^\ast(\mu)=\frac{\sum_{j=1}^\infty p_j\log p_j}{\sum_{j=1}^\infty p_j\log s_j},\]
where $\te{dim}_{\te{H}}^\ast(\mu)$ and $\te{Dim}_{\te{P}}^\ast(\mu)$ represent the Hausdorff and packing dimensions of the measure $\mu$ (see \cite{MR1}).

Let us now give the following lemma.

\begin{lemma} \label{lemma4555}  Let $X$ be a nonempty compact subset of the metric space $\D R^k$. Then, there exists a  finite collection of contractive similarity mappings $T_1, T_2, \cdots, T_K$ on $X$ for some $K\geq 1$ such that $X \sci T_1(X)\uu T_2(X)\uu \cdots \uu T_K(X)$.
\end{lemma}
\begin{proof}
Since $X$ is compact, there exist finite number of open subsets $V_1, V_2, \cdots, V_K$ of $X$ for some  $K\geq 1$, with respect to the relative topology on $X$, such that $X\sci V_1\uu V_2\uu \cdots \uu V_K \sci\ol V_1\uu \ol V_2\uu \cdots \uu\ol V_K$ and $\text{diam}(V_i)< \text{diam}(X)$ for every $i=1,2,\dots ,K.$
Note that to make the diameter of each $V_j$ for $1\leq j\leq K$ small enough, one can take $K$ large enough.
Again, $X$ being compact, each $\ol V_j$ is a compact subset of $X$, and so we can find contractive similarity mappings $T_j$ on $X$ such that
$\ol V_j\sci T_j(X)$ for all $1\leq j\leq K$,
and thus, the lemma is yielded.
\end{proof}

Let us now give the following remark.
\begin{remark} \label{rem11} Each $\ol V_i$ in the proof of the above lemma being compact,  each $\ol V_i$ has finite subcover, and by that way, if needed, one can make the diameter of each $\ol V_i$ small enough. We assume that there exists a positive integer $N$, possibly large, such that for some $1\leq n\leq K$, we have $S_j(X) \sci T_n(X)$ for all $j\geq N+1$.  Then, after some rearrangement in the mappings $T_i$ that appears in Lemma~\ref{lemma4555}, we can take $n=1$. To validate the assumption, if needed, we can also rearrange the infinite mappings $S_1, S_2, S_3, \cdots$.

\end{remark}

In the next sections, we state and prove the main results of the paper.

\section{Quantization dimension of the infinite self-similar measures} \label{sec3}

The following theorem gives the quantization dimension of the infinite self-similar measures with respect to the geometric mean error.
\begin{theorem} \label{main}
Let $\mu$ be the self-similar measure generated by the infinite iterated function system $\set{S_1, S_2, \cdots}$ satisfying the strong separation condition associated with the probability vector $(p_1, p_2, \cdots)$. Assume that the series $\sum_{i=1}^\infty p_i\log s_i$ converges, where $s_i$ are the similarity ratios of $S_i$ for all $i\geq 1$. Then,
\[D(\mu)=D,\]
where $D(\mu)$ is the quantization dimension and $D:=\te{dim}_{\te{H}}^\ast(\mu)$ is the Hausdorff dimension of the measure $\mu$.
\end{theorem}

Let $I^\ast$ denote the set of all words of finite lengths over an alphabet $I$. For any positive integer $N\geq 2$, we call $\gG \sci \set{1, 2, \cdots, N}^\ast$ a \tit{finite maximal antichain} if $\gG$ is a finite set of words in $\set{1, 2, \cdots, N}^\ast$, such that every sequence in $\set{1, 2, \cdots, N}^{\D N}$ is an extension of some word in $\gG$, but no word of $\gG$ is an extension of another word in $\gG$.

To prove the above theorem, we need some lemmas and propositions.
\begin{lemma} \label{lemma001}
Let $\set{S_1, S_2, \cdots, S_{N}}$ be a finite system of contractive similarity mappings with similarity ratios $\set{s_1,  s_2, \cdots, s_{N}}$ associated with a probability vector $(q_1, q_2, \cdots, q_{N})$. Let $\gG\sci\set{1, 2, \cdots, N}^\ast$ be a finite maximal antichain. If $\tilde D$ is the Hausdorff dimension of the corresponding self-similar measure, then
\[\sum_{\gs \in \gG} q_\gs \log s_\gs \leq \frac 1{\tilde D} \sum_{\gs \in \gG} q_\gs \log q_\gs,\]
where $q_\gs=q_{\gs_1}q_{\gs_2}\cdots q_{\gs_k}$ for $\gs=\gs_1\gs_2\cdots\gs_k \in \set{1, 2, \cdots, N}^\ast$, $k\geq 1$, and $q_\es=1$ for the empty word $\es$.
\end{lemma}

\begin{proof}
By \cite[Lemma~1]{GH} (see also \cite[Theorem~3.1]{RS2}), we deduce that $\tilde D \leq \frac{\sum_{i=1}^N q_i\log q_i}{\sum_{i=1}^N q_i\log s_i}$.
Write $\ell(\gG)=\max \set{|\gs| : \gs \in \gG}$. We will prove the lemma by induction on $\ell(\gG)$. If $\ell(\gG)=0$, we know $\gG=\set{\es}$, and so the lemma is obviously true.
 If $\ell(\gG)=1$, then the lemma follows by the definition of $\tilde D$. Next, let $\ell(\gG)=k+1$, and assume that the lemma has been proved for all finite maximal antichains $\gG'$ with $\ell(\gG')\leq k$ for some $k\geq 1$. Define
\begin{align*}
\gG_1&=\set{\gs \in \gG : |\gs| <\ell(\gG)}, \\
\gG_2&=\set{\gs^-  : \gs \in \gG \te{ and } |\gs| =\ell(\gG)},
\end{align*}
and \[\gG_0=\gG_1\UU \gG_2.\]
It is easy to see that $\gG_0$ is a finite maximal antichain with $\ell(\gG_0) \leq k$.  Then, we have
\begin{align*}
a:&=\frac 1 {\tilde D} \sum_{\gs \in \gG} q_\gs\log q_{\gs}-\sum_{\gs \in \gG} q_\gs\log s_{\gs} \\
&=\frac 1 {\tilde D}\Big (\sum_{\gs \in \gG_1} q_\gs\log q_{\gs}+\sum_{\gs \in \gG_2}\sum_{i=1}^N  q_\gs q_i(\log q_\gs+\log q_i)\Big)\\
&\qquad \qquad -\Big (\sum_{\gs \in \gG_1} q_\gs\log s_{\gs}+\sum_{\gs \in \gG_2}\sum_{i=1}^N  q_\gs q_i(\log s_\gs+\log s_i)\Big)\\
&=\frac 1 {\tilde D}\Big [\sum_{\gs \in \gG_1} q_\gs\log q_{\gs}+\sum_{\gs \in \gG_2} q_\gs\log q_\gs +\sum_{\gs \in\gG_2}q_\gs\Big(\sum_{i=1}^N q_i\log q_i\Big)\Big] \\
&\qquad \qquad -\Big [\sum_{\gs \in \gG_1} q_\gs\log s_{\gs}+\sum_{\gs \in \gG_2} q_\gs\log s_\gs +\sum_{\gs \in\gG_2}q_\gs\Big(\sum_{i=1}^N q_i\log s_i\Big)\Big] \\
& \geq \frac 1 {\tilde D}\Big[\sum_{\gs \in \gG_0} q_\gs\log q_{\gs} + \tilde D\sum_{\gs \in\gG_2}q_\gs\Big(\sum_{i=1}^N q_i\log s_i\Big)\Big]- \Big[\sum_{\gs \in \gG_0} q_\gs\log s_{\gs} + \sum_{\gs \in\gG_2}q_\gs\Big(\sum_{i=1}^N q_i\log s_i\Big) \Big]\\
& = \frac 1 {\tilde D}\sum_{\gs \in \gG_0} q_\gs\log q_{\gs}-\sum_{\gs \in \gG_0} q_\gs\log s_{\gs}.
\end{align*}
Since $\gG_0$ is a finite maximal antichain with $\ell(\gG_0)\leq k$, by the induction hypothesis, we have $a \geq 0$, and thus the lemma is proved.
\end{proof}

\begin{note} \label{note1}
We want now to approximate the infinite self-similar measure $\mu$ with discrete measures of finite support.
Let $\C P(X)$ denote the set of all Borel probability measures on the compact set $X \subset \mathbb R^k$. Then,
\begin{align*} d_H (\mu, \gn) :=\sup_{\te{Lip}(f)\leq 1} \Big\{\Big|\int_X f d\mu_1 -\int_X fd\mu_2\Big|\Big\}, \end{align*}
where  $ (\mu, \gn) \in \C M \times \C M$, defines a metric on $\C P(X)$, where $\te{Lip}(f)$ denotes the Lipschitz constant of $f$. Then
$(\C P(X), d_H)$ is a compact metric space (see \cite[Theorem~5.1]{B}).
Further, we know that in the weak topology on $\C P(X)$,
\[\mu_M \to \mu  \Longleftrightarrow \int_X f d\mu_M - \int_X f d\mu \to 0 \te{ for all } f \in \C C(X),\]
where $\C C (X) : =\set{ f : X \to \D R : f \te{ is continuous}}$. It is known that the $d_H$-topology and the weak topology coincide on the space of probabilities with compact support (see \cite{Mat}).  In our case all measures are compactly supported.
Since $X$ is compact, for any Borel probability measure $\gn$ on $X$, we have $\int\|x\|^rd\gn(x)<\infty$. For $r \in (0, +\infty)$ and for two arbitrary  probabilities $\mu_1, \mu_2$, the \tit{$L_r$-minimal metric}, also known as \tit{$L_r$-Wasserstein  metric}, or \tit{$L_r$-Kantorovich metric}, is defined  by the following formula (see for eg. \cite{GL1}):
\[\rho_r(\mu_1, \mu_2)=\inf_\gn \left(\int \|x-y\|^r d\nu(x, y)\right)^{\frac 1 r},\]
where the infimum is taken over all Borel probabilities $\nu$ on $\D R^k \times \D R^k$ with fixed \tit{marginal measures} $\mu_1$ and $\mu_2$, i.e., $\mu_1$ and $\mu_2$ are defined as follows: \[\mu_1(A)=\gn(A\times \D R^k), \te{ and } \mu_2(B)=\gn(\D R^k \times B),\]
for all $A,  B \in \F B$, where $\F B$ is the Borel $\gs$-algebra on $\D R^k$.
Note that the weak topology, the topology induced by $d_H$, and the topology induced by $L_r$-minimal metric $\rho_r$, \textit{all coincide} on the space $\C P(X)$ (see for example \cite{Ru}).
Let us also notice that, for $r=1$, the $\rho_1$ metric is in fact equal to the $d_H$ metric in the compact case, as shown by Kantorovich. For the background on $L_r$-minimal metrics one is refereed \cite{R, RR}.
\end{note}
As a prelude, we write the following lemma. Let us first define for $0<r<+\infty$ and a given Borel probability measure $\mu$ on $\D R^k$, the $n$th quantization error for $\mu$ with respect to the $L_r$-minimal metric by $$e_{n,r}(\mu):=\inf \Big\{\Big(\int d(x, \ga)^r d\mu(x)\Big)^{\frac{1}{r}}: \ga \subset \mathbb{R}^d, \, 1 \leq \text{Card}(\ga) \le n\Big\}.$$
\begin{lemma}[\cite{GL1}, Lemma $3.4$]\label{lem23}
Let $\mathcal{D}_n$ be the set of all discrete probability measures $\gn$ on $X$ such that $|supp(\gn)| \le n.$ Then, for any $\mu \in \mathcal{P}(X)$ we have
\[
e_{n,r}(\mu) = \inf_{\gn \in \mathcal{D}_n} \rho_r(\mu,\gn),
\]
where $\rho_r:\mathcal{P}(X) \times \mathcal{P}(X) \to \mathbb{R}$ is the $L_r$-minimal metric.
\end{lemma}

In the next lemma, we show the continuity of quantization errors $e_{n,r}$ and $e_n.$ In particular, we have
\begin{lemma}\label{use234}
Let $0<r< +\infty,$ and $\mu_N \to \mu $ in the weak topology. Then, for every $n \in \mathbb{N},$ we have
\[
\lim_{N \to \infty} e_{n,r}(\mu_N) =e_{n,r}(\mu) \quad \te{and} \quad \lim_{N \to \infty} e_{n}(\mu_N) =e_{n}(\mu).
\]
\end{lemma}
\begin{proof}
Let $n\geq 1$ be fixed and $\ep>0$ be arbitrary. Then, for $\mu\in \C P(X)$ there exists $\gn_\mu \in \C D_n$ such that $e_{n, r}(\mu)>\rho_r(\mu, \gn_\mu)-\ep$, where $\C  D_n$ has the same meaning as given in Lemma~\ref{lem23}. Clearly, $e_{n, r}(\mu_N)\leq \rho_r(\mu_N, \gn_\mu)$. Since $\rho_r$ is a metric, we have \[e_{n, r}(\mu_N)-e_{n, r}(\mu)< \rho_r(\mu_N, \gn_\mu)-\rho_r(\mu, \gn_\mu)+\ep\leq \rho_r(\mu_N, \mu)+\ep.\]
Since $\ep>0$ is arbitrary, it implies
$e_{n, r}(\mu_N)-e_{n, r}(\mu)\leq \rho_r(\mu_N, \mu).$
Similarly, we can show that
$e_{n, r}(\mu)-e_{n, r}(\mu_N)\leq \rho_r(\mu, \mu_N).$
Thus we deduce that
\[|e_{n, r}(\mu_N)-e_{n, r}(\mu)|\leq \rho_r(\mu_N, \mu).\]
Since $\rho_r(\mu_N, \mu) \to 0$ as $N\to \infty$, the first part of lemma is yielded. Further, using the fact that $\rho_r(\mu_N, \mu)\le \rho_1(\mu_N, \mu)$ for $0< r< 1$, we get
\[|e_{n, r}(\mu_N)-e_{n, r}(\mu)|\leq \rho_1(\mu_N, \mu).\]
Take $r\to 0 $ in the above expression, and then \cite[Lemma~3.5(b)]{GL2} produces
\[|e_{n}(\mu_N)-e_{n}(\mu)|\leq \rho_1(\mu_N, \mu),\]
which completes the proof of the lemma.
\end{proof}

\begin{remark}\label{lemma002}
Let $(\mu_N)$ be a sequence in $\mathcal{P}(X),$ and let $0<r<+\infty.$ Assume that $(\mu_N) $ converges to $\mu \in \mathcal{P}(X)$ with respect to the $L_r$-minimal metric. Then, using Lemma \ref{use234} and Note~\ref{note1}, we have $e_{n,r}(\mu_N) \to e_{n,r}(\mu)$ as $ N \to \infty.$
\end{remark}
Now, we give the following propositions.

\begin{prop} \label{prop112}
Let the self-similar mappings $\set{S_1, S_2, \cdots}$ satisfy the strong separation condition, and let $D=\te{dim}_{\te{H}}^\ast(\mu)$, where $\mu$ is the infinite self-similar measure. Then,
\[\ul D(\mu)\geq D.\]
\end{prop}
\begin{proof}
For any $N\geq 2$, let $L_N=\sum_{j=1}^N p_j$. Let us now consider the sequence $\set{t_N}_{N\geq 2}$, where
\[t_N=\frac{\sum_{j=1}^N \frac {p_j}{L_N} \log \frac{p_j}{L_N}}{\sum_{j=1}^N \frac {p_j}{L_N} \log s_j}.\]
Let $\mu_N$ be the self-similar measure generated by the finite system of contractive similarity mappings $S_1, S_2, \cdots, S_N$ associated with the probability vector $(\tilde p_1, \tilde p_2, \cdots, \tilde p_N)$, where
$\tilde p_i =\frac{p_i}{L_N} \te{ for } 1 \leq i \leq   N$.
Under the strong separation condition, we know that
\[\te{dim}_{\te{H}}^\ast(\mu_N)=\frac{\sum_{i=1}^N \tilde p_i\log \tilde p_i}{\sum_{i=1}^N \tilde p_i\log s_i}=t_N,\]
where $\te{dim}_{\te{H}}^\ast(\mu_N)$ is the Hausdorff dimension $t_N$ of the self-similar measure $\mu_N$ (see \cite {GH}).
Note that $\lim_{N\to\infty} t_N= D$, where $D=\te{dim}_{\te{H}}^\ast(\mu)$ is the Hausdorff dimension of the infinite self-similar measure $\mu$.
By \cite{GL2}, under the strong separation condition, we know that the quantization dimension of the self-similar measure $\mu_N$ with respect to the geometric mean error exists and equals the Hausdorff dimension $t_N$ of the measure $\mu_N$. With respect to the $L_r$-minimal metric, it is known that $\set{\mu_N}_{N\geq 2}$ tends to the probability measure $\mu$ as $N\to \infty$ (see \cite[Theorem 3]{N}). Thus, for any discrete $\ga \sci \D R^k$, the function $\log d(x, \ga)$ being continuous for $\mu$ almost every $x\in X$, we have
\[\int \log d(x, \ga)d\mu_N \to \int \log d(x, \ga) d\mu \te{ as } N\to \infty.\]
Using the above fact and Lemma \ref{use234}, it follows that for any $n\in \D N$, $\hat e_n(\mu_N) \to \hat e_n(\mu)$ as $ N \to \infty$.
By \cite[Lemma~5.10]{GL2}, we know that for the probability measure $\mu_N$ there is an $n_0\in \D N$ such that
\[\hat e_n(\mu_N)=\sum_{i=1}^N \tilde p_i\log s_i + \min\Big\{\sum_{i=1}^N \hat p_i \hat e_{n_i}(\mu_N) : n_i\geq 1, \, \sum_{i=1}^N n_i\leq n\Big\}\]
for all $n\geq n_0$.
Let $c_N=\min \set{\frac 1 {t_N}\log n+\hat e_n(\mu_N) : n\leq n_0}$ and $c=\min \set{\frac 1 D \log n+\hat e_n(\mu) : n\leq n_0}$. Since both $\hat e_n(\mu_N)$ and $\hat e_n(\mu) >-\infty$ for all $n\in \D N$, we have $c$ and $c_N>-\infty$. Moreover, as $\hat e_n(\mu_N) \to \hat e_n(\mu)$,  we have $c_N\to c$ whenever $N\to \infty$. As shown in the proof of Theorem~5.11 in \cite{GL2}, we have
\[\hat e_n(\mu_N) \geq c_N-\frac{1}{t_N} \log n,\]
for all $n\in \D N$. Now taking $N\to \infty$, we have
\[\hat e_n(\mu) \geq c-\frac{1}{D} \log n.\]
This implies
\[\inf_{n\in \D N} \Big (\log n+ D \hat e_n(\mu)\Big) \geq c D>-\infty,\]
which by Proposition~\ref{prop1} yields that $\ul D(\mu)\geq D$, and hence the proposition.
\end{proof}

\begin{prop} \label{prop113} Let the infinite system of self-similar mappings $S_j$ with similarity ratios $s_j$, $j\in\D N$, associated with the probability vector $(p_1, p_2,  \cdots)$ satisfy the strong open set condition, and let the series $\sum_{j=1}^\infty p_j \log s_j$ converge. Then, $\ol D(\mu) \leq D$, where $D$ is the Hausdorff dimension $\te{dim}_{\te{H}}^\ast(\mu)$ of the infinite self-similar measure $\mu$.
\end{prop}

\begin{proof}
 By Remark~\ref{rem11}, there exists a positive integer $N$ such that $S_j(X)$ are contained in $T_1(X)$ for all $j \geq  N+1.$
Let us now define the finite system of contractive similarities $\tilde S_i$ for $1 \leq i \leq N+1$, such that $\tilde S_i = S_i$ for $1 \le i \le N$ and $\tilde S_{N+1} = T_1$.
Let
\[\tilde p_i = p_i \te{ for } 1 \leq i \leq   N, \te{ and }  \tilde p_{N+1} = \mathop{\sum}\limits_{i >N} p_i.\]
Further, let $\tilde s_i$ be the contraction ratio of $\tilde S_i$, for $ 1 \le i \le N+1$.
Now from the self-similarity condition of the measure $\mu$, we have the decomposition
\begin{equation}\label{dec}
\mu = \mathop{\sum}\limits_{j \ge 1} p_j \mu\circ S_j^{-1} = \mathop{\sum}\limits_{j=1}^N p_j \mu\circ S_j^{-1} + \mathop{\sum}\limits_{j >N} p_j \mu\circ S_j^{-1}.
\end{equation}
By $\set{1, 2, \cdots, N+1}^\ast$ we denote the set of all sequences of finite lengths over the alphabet $\set{1, 2, \cdots, N+1}^\ast$.
For  $\omega = (\omega_1, \cdots, \omega_k) \in \{1, \cdots, N+1\}^*, \ k\ge 1$, set $\tilde p_\omega:= \tilde p_{\go_1} \cdots \tilde p_{\omega_k}$. Let $\omega^- = (\omega_1, \cdots, \omega_{k-1})$ be the truncation of $\omega$ obtained by cutting the last letter of $\go$.
Let $\tilde p_{\min}=\min\set{\tilde p_j: 1\leq j\leq N+1}$. Consider an arbitrary integer $n\in \D N$, such that $\frac 1 n <\tilde p_{\min}^{2}$.  Write $\ep_n=\frac 1 n \tilde p_{\min}^{-1}$.
Then $0<\ep_n<1$, write
\[
F_n =\set{\go \in \set{1, 2, \cdots, N+1}^\ast : \tilde p_{\go^-} \geq \ep_n>\tilde p_\go}.
\]
Then, $F_n$ is a finite maximal antichain.
Note that if $\go \in F_n$, then $\tilde p_\omega \geq \frac 1n$. Thus, we have
\[1=\mathop{\sum}\limits_{\omega \in F_n} \tilde p_\omega  \ge \te{Card}(F_n) \cdot \frac{1}{n},\]
which implies $\te{Card}(F_n)\leq n$.
Let $B$ be a set of cardinality $n$, which has points in each of the sets $\tilde S_\omega(X)$ for $\omega \in F_n$; this is possible since, we have seen $\te{Card}(F_n) \le n$. Moreover, the support of $\mu$ is compact, and so there exists a constant $C>0$, such that
\begin{align*}
\hat e_n(\mu)\leq \int \log d(x, B) d\mu\leq \sum_{\go\in F_n} \tilde p_\go\int \log d(x, B) d(\mu\circ \tilde S_\go^{-1})\leq \sum_{\go\in F_n}\tilde p_\go \log \tilde s_\go +C.
\end{align*}
Let $\tilde D_N$ be the Hausdorff dimension of the self-similar measure generated by the mappings $\set{\tilde S_1, \cdots, \tilde S_{N+1}}$ with contractive ratios $\set{\tilde s_1, \cdots, \tilde s_{N+1}}$ and the probability vector $(\tilde p_1,\cdots, \tilde p_{N+1})$. Then, by Lemma~\ref{lemma001}, we have
\begin{align*}
\hat e_n(\mu) \le \frac 1{\tilde D_N} \sum_{\go \in F_n}\tilde p_{\go} \log \tilde p_{\go} +C,
\end{align*}
which is true for all but finitely many $n\in \D N$. Since $\ep_n=\frac 1 n \tilde p_{\min}^{-1}$, we have
\begin{equation} \label{eq345}
\log n+\tilde D_N \hat e_n(\mu) \leq \sum_{\go\in F_n}\tilde p_{\go} \log \tilde p_{\go}-\log \ep_n-\log\tilde p_{\min} +C\tilde D_N,
\end{equation}
for all but finitely many $n\in \D N$.
Since $\tilde p_{\go} < \ep_n$ for all $\go \in F_n$, we have $\log \ep_n \geq \log\tilde p_{\go}$, and hence,
\[
\sum_{\go \in F_n} \tilde p_{\go} \log \tilde p_{\go} \leq \log \ep_n,
\]
which yields
\[
\limsup_{n \to \infty} \Big (\sum_{\go \in F_n} \tilde p_{\go} \log \tilde p_{\go}-\log \ep_n\Big )\leq 0.
\]
Thus, \eqref{eq345} implies that
\[
\limsup_{n\to \infty}\Big(\log n+\tilde D_N \hat e_n(\mu)\Big) \leq -\log\tilde p_{\min} +C\tilde D_N <+\infty,
\]
which yields $\ol D(\mu) \leq \tilde D_N$. Given $\sum_{j=1}^\infty p_j \log s_j$ converges, and so by \cite[Lemma~3.1]{MR1}, $\sum_{j=1}^\infty p_j \log p_j$ converges, which implies
\[
 \lim_{N\to \infty}\frac{\sum_{j=1}^N \tilde p_j \log \tilde p_j}{\sum_{j=1}^N \tilde p_j \log \tilde s_j}=\frac {\sum_{j=1}^\infty p_j \log p_j}{\sum_{j=1}^\infty p_j \log s_j}
\]
exists and equals $D$.
Thus, since $\ol D(\mu) \leq \tilde D_N \leq \frac{\sum_{j=1}^N \tilde p_j \log \tilde p_j}{\sum_{j=1}^N \tilde p_j \log \tilde s_j}$ holds for arbitrarily large $N$, taking $N\to \infty$ we have $ \ol D(\mu) \leq D$. Hence the proposition.
\end{proof}

\subsection*{Proof of Theorem~\ref{main}} Recall Proposition~\ref{prop1}. Then, Proposition~\ref{prop112} tells us that $\ul D(\mu)\geq D$, and Proposition~\ref{prop113} implies that  $\ol D(\mu) \leq D$. Since, $D\leq \ul D(\mu)\leq \ol D(\mu)\leq D$, the proof of Theorem~\ref{main} is complete.
\qed

\section{Stability of quantization dimension} \label{sec4}
In this section, our aim is to study the stability of quantization dimensions for the infinite systems with respect to the geometric mean error.
Let us start this section with the following example, which conveys the reader that quantization dimension is not continuous in general.
\begin{example}
Let $\delta_t$ be the Dirac measure at $t \in \mathbb{R}$, and $\mathcal{L}^1$ denote the one-dimensional Lebesgue measure on $\mathbb{R}.$
Define a sequence of measures $(\mu_n)$ as follows: $\mu_n= \frac{1}{n} \sum_{i=1}^n \delta_{i/n}.$ From the very construction of $\mu_n$, it follows that $\mu_n \to \mathcal{L}^1|_{[0,1]}.$ However, we prove that $D(\mu_n)$ does not converge to $D\big(\mathcal{L}^1|_{[0,1]}\big).$ Note that $\mu =\mathcal{L}^1|_{[0,1]},$ and hence $D\big(\mathcal{L}^1|_{[0,1]}\big)=1.$ It is simple to show that $D(\mu_n)=0$ for all $n\in \D N.$ Therefore, $D(\mu_n)$ does not converge to $D(\mu)$.
\end{example}
In the sequel, Euclidean norm on $\mathbb{R}^k$ is denoted by $\|\cdot\|$, and if $S: X \to X$ is a mapping, then $\|S\|$ represents the supremum norm on $X$, i.e.,  $\|S\|:=\sup_{x\in X}\|S(x)\|$.

\begin{lemma} \label{lemma45}
Let $\mu_n$ be the Borel probability measure generated by the infinite IFS give by $\{X; S_{n, j},~ j \in \mathbb{N}\}$ associated with the probability vector $(p_{n,1}, p_{n, 2}, \cdots )$. Let $\mu$ be the Borel probability measure generated by the IFS given by $\{X; S_{j}, ~ j \in \mathbb{N}\}$ associated with the probability vector $(p_{1}, p_{2}, \cdots)$, i.e.,
\[\mu_n=\sum_{j=1}^{\infty} p_{n, j}\mu_n\circ S_{n, j}^{-1}, \te{ and } \mu=\sum_{j=1}^{\infty} p_j \mu\circ S_j^{-1}.\]
Assume that there exists $x_0 \in X$ such that $\sup_{j \in \mathbb{N}} \|S_j(x_0)\|< \infty$,  $\sum_{j=1}^{\infty}\|S_{n, j} - S_j \|\to 0 $, and $\sum_{j=1}^{\infty}|p_{n, j} - p_j| \to 0$ as $n\to \infty$.
Then, $\mu_n\to \mu$ as $n\to \infty$.
\end{lemma}
\begin{proof}
We have
\begin{equation*}
\begin{aligned}
& d_H(\mu_n, \mu) =\sup_{\te{Lip}(g)\leq 1} \Big|\int g d\mu_n -\int g d\mu\Big|\\
&=\sup_{\te{Lip}(g)\leq 1} \Big| \sum_{j=1}^{\infty} p_{n, j}\int (g\circ S_{n, j}) d\mu_n- \sum_{j=1}^{\infty}  p_j \int (g\circ S_j) d\mu\Big|\\
&\leq \sum_{j=1}^{\infty} \sup_{\te{Lip}(g)\leq 1} \Big| p_{n, j}\int (g\circ S_{n, j}) d\mu_n-p_j \int (g\circ S_j) d\mu\Big|\\
&=\sum_{j=1}^{\infty} \sup_{\te{Lip}(g)\leq 1} \Big| p_{n, j}\int ((g\circ S_{n, j})-(g\circ S_j)) d\mu_n+(p_{n, j}-p_j) \int (g\circ S_j) d\mu\Big|\\
&\leq \sum_{j=1}^{\infty} \sup_{\te{Lip}(g)\leq 1} \Big| p_{n, j}\int ((g\circ S_{n, j})-(g\circ S_j)) d\mu_n\Big|+\sum_{j=1}^{\infty} \sup_{\te{Lip}(g)\leq 1}\Big|(p_{n, j}-p_j) \int (g\circ S_j) d\mu\Big|\\
&\leq \sum_{j=1}^{\infty} \sup_{\te{Lip}(g)\leq 1} p_{n, j}\int \|(g\circ S_{n, j})-(g\circ S_j)\| d\mu_n+\sum_{j=1}^{\infty} \sup_{\te{Lip}(g)\leq 1}|(p_{n, j}-p_j|  \int\|g\circ S_j\| d\mu.
\end{aligned}
\end{equation*}
Since for all $ j \in \mathbb{N}$, $ S_j$ are contraction mappings, we get
 \begin{align*} \sup_{x \in X }\|S_j(x)\|&\le \sup_{x \in X }\|S_j(x)-S_j(x_0)\|+\|S_j(x_0)\| \le s_j \sup_{x \in X} d(x,x_0)+\|S_j(x_0)\|\\
 &\le s_j \text{diam}(X)+\|S_j(x_0)\|.
 \end{align*}
Further, using the fact that $\sup_{j\in \mathbb{N}} s_j  < 1,$ and the assumption of pointwise boundedness of $\{S_{j}, ~ j \in \mathbb{N}\}$, we choose a constant $M_*>0$, such that $\|g\circ S_j\|\leq M_*$ for all $ j \in \mathbb{N}$.
Since $\sum_{j=1}^{\infty}\|S_{n, j} - S_j \|\to 0 $, and $\sum_{j=1}^{\infty}|p_{n, j} - p_j| \to 0$, for any given $\ep>0$, as $\frac \ep{M_*2^{j+1}}>0$, there exists a positive integer $\tilde N(\ep)$, such that for all $n\geq \tilde N(\ep)$, we have
\[
\|S_{n, j}-S_j\|<\frac{\ep}{2^{j+1}}, \te{ and } |p_{n, j}-p_j|<\frac{\ep}{M_*2^{j+1}},
\]
for all $ j \in \mathbb{N}$.
Again, if $\te{Lip}(g)\leq 1$, then for all $n\geq \tilde N(\ep)$, we have $\|(g\circ S_{n, j})-(g\circ S_j)\|\leq \|S_{n, j}-S_j\|<\frac{\ep}{M_*2^{j+1}}$.
Thus, for the above $\ep$, and $\tilde N(\ep)$, if $n\geq \tilde N(\ep) $, we have
\begin{equation*}
\begin{aligned}
&d_H(\mu_n, \mu) <\sum_{j=1}^{\infty} p_{n, j}\frac{\ep}{2^{j+1}}+\sum_{j=1}^{\infty} \frac{\ep}{M_*2^{j+1}} M_*=\ep,
\end{aligned}
\end{equation*}
which completes the proof.
\end{proof}
 In the next theorem, we show the stability of the quantization dimension with respect to the geometric mean error.
\begin{theorem} \label{theoremlast}
Let $\mu_n$, and $\mu$ be the Borel probablity measures given by the IFSs as given in the above lemma, that is,
\[
\mu_n=\sum_{j=1}^{\infty} p_{n, j}\mu_n\circ S_{n, j}^{-1}, \te{ and } \mu=\sum_{j=1}^{\infty} p_j \mu\circ S_j^{-1}.
\]
Assume that all IFSs satisfy the hypotheses as in the above lemma and Theorem \ref{main}. Further, assume that $\inf_{j,n \in \mathbb{N}} \{s_{n,j},s_j\}>0$ and for each $N \in \mathbb{N}$, $\inf_{n \in \mathbb{N}, ~j\in  \set{1,2,\cdots, N}} \{p_{n,j},p_j\}>0.$
Then, $D(\mu_n) \to D(\mu)$ as $n\to \infty.$
\end{theorem}
\begin{proof}
In view of Lemma \ref{lemma45}, we note that $\mu_n \to \mu.$ By Theorem \ref{main},
\[
D(\mu_n)=\frac{\sum_{j=1}^\infty p_{n,j}\log p_{n,j}}{\sum_{j=1}^\infty p_{n,j}\log s_{n,j}}, \quad \te{and} \quad D(\mu)=\frac{\sum_{j=1}^\infty p_j\log p_j}{\sum_{j=1}^\infty p_j\log s_j}.
\]
Now, thanks to the triangle inequality:
\[
|p_{n,j}\log s_{n,j} - p_j\log s_j | \le |\log s_{n,j}||p_{n,j}- p_{j}|+p_{j} |\log s_{n,j} - \log s_{j}|.
\]
Consequently, using the mean-value theorem, we have
\[
\sum_{j=1}^\infty|p_{n,j}\log s_{n,j} - p_j\log s_j | \le |\log \big(\inf_{j,n \in \mathbb{N}} \{s_{n,j},s_j\}\big)\big| \sum_{j=1}^\infty |p_{n,j}- p_{j}|+ \frac{1}{\inf_{j,n \in \mathbb{N}} \{s_{n,j},s_j\}}\sum_{j=1}^\infty | s_{n,j} -  s_{j}|.
\]
Since $\sum_{j=1}^{\infty}\|S_{n, j} - S_j \|\to 0 $ and  $\sum_{j=1}^{\infty}|p_{n, j} - p_j| \to 0,$ the above expression yields
\begin{equation}\label{lasteqn3}
  \sum_{j=1}^\infty|p_{n,j}\log s_{n,j} - p_j\log s_j | \to 0, \te{ i.e., } \sum_{j=1}^\infty p_{n,j}\log s_{n,j} \to   \sum_{j=1}^\infty p_j\log s_j \te{ as } n \to \infty.
\end{equation}
Again, using the mean-value theorem, we have
\begin{equation*}
\sum_{j=1}^\infty|p_{n,j}\log p_{n,j} - p_j\log p_j | \le\sum_{j=1}^\infty |p_{n,j}- p_{j}|+ \sum_{j=1}^\infty |\log \big(\min \{p_{n,j},p_j\}\big)\big| |p_{n,j}- p_{j}|.
\end{equation*}
To prove $\sum_{j=1}^\infty|p_{n,j}\log p_{n,j} - p_j\log p_j | \to 0$ as $n \to \infty$, it suffices to show that
\[
\sum_{j=1}^\infty |\log \big(\min \{p_{n,j},p_j\}\big)\big| |p_{n,j}- p_{j}| \to 0 ~~~~\te{as}~~~ n \to \infty.
\]
 Suppose it is not true. Then, there exist a real number $\epsilon_0> 0$ and a sequence $(n_k)$ such that
\begin{equation}\label{new71}
\sum_{j=1}^\infty |\log \big(\min \{p_{n_k,j},p_j\}\big)\big| |p_{n_k,j}- p_{j}| \geq \ep_0 ~~~~~\te{for all } ~~~k \in \mathbb{N}.
\end{equation}
Hence, there exist $ N_*, N_0 \in \mathbb{N}$ with $N_\ast<N_0$  such that
\[
\big|\log \big(\min_{j\in \set{N_*,N_*+1,\cdots,N_0}} \{p_{n_k,j},p_j\}\big)\big|\sum_{j=N_*}^{N_0}  |p_{n_k,j}- p_{j}| \geq \sum_{j=N_*}^{N_0} |\log \big(\min \{p_{n_k,j},p_j\}\big)\big| |p_{n_k,j}- p_{j}| \geq \frac{\ep_0}{2},
\]
for every $k \in \mathbb{N}.$
This immediately gives
\begin{equation}\label{lasteqn1}
    \sum_{j=N_*}^{N_0}  |p_{n_k,j}- p_{j}|     \geq \frac{\ep_0}{2\big|\log\big(\inf_{n \in \mathbb{N},~j\in \set{1,2,\cdots, N_0}}\{p_{n,j},p_j\}\big)\big|}  \te{ for all } k \in \mathbb{N}.
\end{equation}
Since $\sum_{j=1}^{\infty}  |p_{n,j}- p_{j}| \to 0$ as $n \to \infty,$ for $\epsilon= \frac{\ep_0}{2\big|\log\big(\inf_{n \in \mathbb{N},~j\in \set{1,2,\cdots, N_0}}\{p_{n,j},p_j\}\big)\big|} $, there exists $N_1 \in \mathbb{N}$ such that
\[
\sum_{j=1}^{\infty}  |p_{n,j}- p_{j}| < \ep \te{ for all } n \geq N_1,
\]
which contradicts \eqref{lasteqn1}. Hence,
\begin{equation} \label{lasteqn2}
\sum_{j=1}^\infty|p_{n,j}\log p_{n,j} - p_j\log p_j | \to 0, \te{ i.e., }  \sum_{j=1}^\infty p_{n,j}\log p_{n,j} \to \sum_{j=1}^\infty p_j\log p_j \te{ as } n\to \infty
\end{equation}
Thus, by \eqref{lasteqn3} and \eqref{lasteqn2}, the proof of the theorem is complete.
\end{proof}

\begin{remark}
Let us write an example of probability vectors satisfying the condition given in Theorem~\ref{theoremlast}, that is, for each $N \in \mathbb{N}$, $\inf_{n \in \mathbb{N},~j\in \set{1,2,\cdots, N}} \{p_{n,j},p_j\}>0.$
For the following example, the condition is satisfied:
\[ p_{ n, j }= \frac{1}{ 2^j}, \te{ and } p_{j}= \frac{1}{2^j},\]
where $n, j\in \D N$.
 However, the assumption does not hold for arbitrary choices of infinite probability vectors. For example, let
 $(p_{ n, j })= \Big(\frac{1}{n+2}, \frac{1}{n+2}, t, t^2,...\Big),$ where $t$ is a positive real number such that $\sum_{j=1}^{\infty}p_{ n, j }=1,$
and let $(p_j)$ be an infinite probability vector. Then, for $N=2,$ we have
\[
\inf_{ n\in \mathbb{N},~j=1,2} \{ p_{ n , j }, p_j \} = 0.
\]
\end{remark}


\end{document}